\documentclass
{amsart}
\usepackage{amssymb,mathrsfs}
\usepackage{amssymb}
\usepackage{amsmath}
\usepackage{amsfonts, mathrsfs}
\usepackage{mathtools}
\usepackage{epsfig}
\usepackage{color}
\usepackage{epstopdf}
\usepackage{graphicx}
\usepackage{srcltx}
\usepackage{amsthm}
\usepackage[mathscr]{eucal}
\usepackage{enumerate}
\usepackage{verbatim}
\usepackage[draft]{todonotes}
\usepackage{relsize}
\usepackage{kotex}
\usepackage[
]
{hyperref}
\usepackage{cases} 
\usepackage{empheq}
\usepackage{subfigure}
\usepackage{tikz}
\usepackage{array}
\usepackage{tabularx}
\newcolumntype{C}[1]{>{\centering\arraybackslash}p{#1}}

\usetikzlibrary{arrows,automata}
\usepackage{varwidth}
\usepackage{tikz}
\usetikzlibrary{backgrounds,patterns,calc}
\usepackage{xparse}

\usepackage[normalem]{ulem}
\newcommand{\stkout}[1]{\ifmmode\text{\sout{\ensuremath{#1}}}\else\sout{#1}\fi}

\hypersetup{
	colorlinks=false,       
	linkcolor=blue,          
	citecolor=red,        
	filecolor=magenta,      
	urlcolor=cyan           
}

\newcommand{\Be}{\begin{equation}}
\newcommand{\Ee}{\end{equation}}
\newcommand{\Bea}{\begin{eqnarray}}
\newcommand{\Eea}{\end{eqnarray}}
\newcommand{\Bel}{\begin{align}}
\newcommand{\Eel}{\end{align}}
\newcommand{\Beas}{\begin{eqnarray*}}
	\newcommand{\Eeas}{\end{eqnarray*}}
\newcommand{\Benu}{\begin{enumerate}}
	\newcommand{\Eenu}{\end{enumerate}}
\newcommand{\Bi}{\begin{itemize}}
	\newcommand{\Ei}{\end{itemize}}

\numberwithin{equation}{section}
\newcommand{\supp} {\text{supp\! }}

\theoremstyle{plain}
\newtheorem{thm}{Theorem}[section]
\newtheorem{cor}[thm]{Corollary}
\newtheorem{lem}[thm]{Lemma}
\newtheorem{prop}[thm]{Proposition}

\theoremstyle{remark}
\newtheorem{rmk}{Remark}  

\theoremstyle{definition}
\newtheorem{defn}[thm]{Definition}

\definecolor{ao}{rgb}{0, 0.5, 0}

\newcommand{\R}{\mathbb R}

\usepackage{tikz}
\usepackage{pgfplots}
\pgfplotsset{compat=1.13,
tick label style={color=white},
  label style={font=\small},
  legend style={font=\small}}

\usepackage{mathtools}

\title[unions of curves]
{Remarks on dimension of unions of curves}

\author{Seheon Ham}
\author{Hyerim Ko}
\author{Sanghyuk Lee}
\author{Sewook Oh}

\thanks{}
\address{Department of Mathematical Sciences and RIM, Seoul National University, Seoul 08826, Republic of Korea}
\email{seheonham@snu.ac.kr}
\email{kohr@snu.ac.kr}
\email{shklee@snu.ac.kr}
\email{dhtpdnr0220@snu.ac.kr}

\begin{document}

\begin{abstract}
We study an analogue of Marstrand's circle packing problem for curves in higher dimensions. 
We consider collections of curves which are generated by translation and dilation of a curve $\gamma$ in $\mathbb R^d$, i.e., $ x + t \gamma$, $(x,t) \in \mathbb R^d \times (0,\infty)$. 
For a Borel set $F \subset  \mathbb R^d\times (0,\infty)$, we show the unions of curves $\bigcup_{(x,t) \in F} ( x+t\gamma )$ has Hausdorff dimension  at least $\alpha+1$ whenever $F$ has Hausdorff dimension bigger than $\alpha$, $\alpha\in (0, d-1)$.   We also 
obtain results for unions of curves generated by multi-parameter dilation of $\gamma$. 
One of the main ingredients is a local smoothing type estimate (for averages over curves) relative to fractal measures. 
\end{abstract}

\maketitle

\section{Introduction}

Let $\gamma $ be a smooth curve from $I_\circ =[0,1]$ to $\mathbb R^d$, $d\ge2$. For a given $\gamma$, we consider the curves given  by translation and dilation of $\gamma(I):=
\{\gamma(s): s\in I\subset \subset I_\circ\}$.
For a Borel set $F \subset \mathbb R^d\times (0,\infty)$, denote  
\begin{equation}\label{union-curve} 
\Gamma(F) :=  \bigcup_{(x,t) \in F} \big( x+t\gamma(I) \big), 
\end{equation}
where $(x,t)\in \mathbb R^d\times (0,\infty).$  In this paper, we are concerned with the problem of determining how small the dimension of $\Gamma(F)$ can be depending on that  of $F$. 

This problem has a rich history when $\gamma$ is the unit circle in $\mathbb R^2$. 
Besicovitch--Rado \cite{BR} and Kinney \cite{Ki} showed that there exists a set 
(so-called BRK set) of Lebesgue measure zero which contains a circle of every diameter, i.e., $\Gamma(F)$ with $F=\{ (X(t),t): t \in (0,\infty)\}$ for some function  $X:\mathbb R^+ \rightarrow \mathbb R^2$
(see also \cite{Da, KW} for further generalization). In the same vein, it is natural to ask whether or not there exists a plane set of measure zero containing circles with centers at every point, i.e., $\Gamma(F)$ with $F = \{ (x, T(x)): x \in \mathbb R^2\}$ for a function $T:\mathbb R^2\rightarrow \mathbb R^+$ (see \cite{Le} or \cite[p.105]{Fa}).
Marstrand \cite{Marstrand} showed that such a set cannot be of Lebesgue measure zero. 

Those problems for the circle in $\mathbb R^2$ generalize naturally with subsets $F \subset \mathbb R^2 \times (0,\infty)$.
Let ${\rm Proj}_x F$ be the orthogonal projection of $F$ onto $x$-space. 
It was proved in \cite{Marstrand} that if ${\rm Proj}_x F$ has positive Lebesgue measure, then so does $\Gamma(F)$
 (see also \cite{Bourgain}). 
Later, Mitsis \cite{Mitsis} showed  that it is sufficient for $\Gamma(F)$ being of positive measure  that  $\dim_H ({\rm Proj}_x F) > 3/2$. Here, $\dim_H $ denotes the Hausdorff dimension. 
It was also conjectured that $\Gamma(F)$ has positive Lebesgue measure  if  $\dim_H  ({\rm Proj}_x F) >1$. This condition  is optimal  in view of Talagrand's construction \cite{Talagrand}  of   a measure zero set which contains a circle centered at every point on a straight line. 
The conjecture was later proved by Wolff \cite{Wo00}.
In fact, a stronger result was obtained: 
For any  set $F\subset \mathbb R^2\times(0,\infty)$ satisfying $\dim_H F >1$ (not necessarily $\dim_H ({\rm Proj}_xF)>1$), any set containing $\Gamma(F)$ can not be of Lebesgue measure zero.

The main object of  this paper is to study an extension of the aforementioned problem to space curves in $\mathbb R^d$, $d\ge3$.
Similar problems for non-planar curves are less well understood (see \cite{PS,  KLO} for partial results).
We assume that $\gamma$ is nondegenerate, namely, 
\begin{equation}\label{torsion}
\det \begin{pmatrix} \gamma'(s) & \cdots & \gamma^{(d)}(s) \end{pmatrix} \neq 0, \quad \forall s\in I_\circ. 
\end{equation}

The next theorem is our first main result, which generalizes the circle packing problem to nondegenerate curves in $\mathbb R^d$, $d\ge3$.

\begin{thm}\label{thm-main}
Let $d \ge 3$ and $\gamma$ be a nondegenerate curve in $\mathbb R^d$. 
Suppose that a Borel set $F$ satisfies $\dim_H F > \alpha$ for some $0<\alpha \le d-1 $. Then, 
\begin{equation}\label{dimE} 
\dim_H E \ge \alpha+1
\end{equation}
holds whenever 
 a Borel set $E\subset \mathbb R^d$ contains $\Gamma(F)$. 
\end{thm}

In $\mathbb R^2$, K\"{a}enm\"{a}ki--Orponen--Venieri \cite{KOV} provided a geometric proof of \eqref{dimE} when $F $ is an analytic set. 

Theorem \ref{thm-main} is optimal in that $\dim_H(\Gamma(F)) \le \dim_H F+1$ for a Borel family $\Gamma(F)$, as can be easily seen by the standard covering argument.
The case that $\Gamma(F)$ has positive Lebesgue measure was of special interest.
In \cite[Corollary 1.6]{KLO}, it was shown that
$\Gamma(F)$ has positive Lebesgue measure if $\dim_H F >d-1$. Conversely, 
modifying  Kinney's example \cite{Ki}, one can show that the condition  $\dim_H F > d-1$ is necessary, in general,  for $\Gamma(F)$ to have positive Lebesgue measure. 
See Proposition \ref{BRK}. 

Our result also extends to curves of different type. 
We say that $\gamma$ is of finite type if for each $s \in I_\circ$, there exists a $d$-tuple of positive integers $(a_1,\dots, a_d)$ such that 
\[\det \begin{pmatrix} \gamma^{(a_1)}(s) & \cdots & \gamma^{(a_d)}(s) \end{pmatrix}  \neq 0.\] 
Theorem \ref{thm-main} also continues to hold for finite type curves, since a finite type curve contains a nondegenerate sub-curve.
See \cite{PS, KLO} for a previous result for finite type curves.   One can easily see that Theorem \ref{thm-main} does not hold in general if the finite type condition is not satisfied.
For example, consider the curve $\gamma$  contained in a $k$-dimensional linear subspace $V_k \subset \mathbb R^d$ for $2\le k \le d-1$ such that $\gamma'(s),\dots,\gamma^{(k)}(s)$ spans $V_k$ for all $s\in I_\circ$. 
For $\varepsilon \in (0,1/2)$, let  $F_k \subset V_k \times (0,\infty)$ 
such that $\dim_H F_k = k-1+ 2\varepsilon$. 
Then $\dim_H \Gamma(F_k) \le  k$ since $x_k + t\gamma(I) \subset V_k$ for any $(x_k,t)\in F_k$.
Let $G_k \subset V_k^\perp$ such that $\dim_H G_k =\overline{\dim}_BG_k= 1-\varepsilon$. 
Here, $\overline{\dim}_B$ denotes the upper box counting dimension. 
Clearly, $\dim_H (F_k + G_k) > k$.
If we take $F= F_k + G_k$, it follows that $\dim_H \Gamma(F)\le \dim_H \Gamma(F_k) + \dim_H {G_k} < k+ 1 $ even if $\dim_H F > k$.

\subsection*{Averages over curves} Our proof of Theorem \ref{thm-main} relies on local smoothing estimates  relative to fractal measures for averaging operators rather than geometric approach used in some of the previous works (see \cite{Wo97,Mitsis, Za, KOV}).

We consider an averaging operator given by 
\[
\mathcal A_\gamma f(x,t) =\chi(t) \int  f(x + t\gamma(s)) \psi(s)\,ds,
\]
where $\chi \in C_0^\infty((1/2,4))$ and $\psi$ is a nonnegative smooth function satisfying $\psi=1$ on $I$ and $\supp \psi \subset I_\circ$.

\begin{defn}
Let  $\mathbb B^{d} (z,\rho)$ denote  the ball of radius $\rho$ centered at $z$ in $\mathbb R^d$. 
For a non-negative Borel measure $\mu$ defined on $\R^{d}$ we say $\mu$ is  $\alpha$-dimensional  if  there is a constant $C_\mu$
 such that
\begin{equation*}
\mu(  \mathbb B^{d}(z,\rho) ) \le C_\mu \rho^\alpha, \quad \forall  (z,\rho) \in \mathbb R^{d}\times \mathbb R^+
\end{equation*} 
for some $\alpha \in (0,d]$.  We denote by $\mathfrak C^{d}(\alpha)$ the set of all $\alpha$-dimensional measures in $\mathbb R^{d}$.
For $\mu \in \mathfrak C^{d}(\alpha)$, we  set
\[
\langle \mu\rangle_\alpha := \sup_{z\in \mathbb R^{d},\, \rho>0} \rho^{-\alpha} 
\mu(  \mathbb B^{d} (z,\rho) ). 
\]
\end{defn}

It is well-known that Marstrand's results \cite{Marstrand} can be deduced from $L^p$, $p\neq \infty$,  estimate for the circular maximal function.
In \cite{HKL}, some of the authors extended the $L^p$ circular maximal estimate to that relative to $\alpha$-dimensional measures, which recovers the aforementioned Wolff's result in \cite{Wo00}.
Moreover, $L^p$ maximal bound was extended to space curves (see \cite{PS,KLO1, BGHS, KLO}). 
Various forms of local smoothing estimates played important roles in proving those results. 
Similarly, to prove Theorem \ref{thm-main}, we make use of local smoothing estimate relative to $\alpha$-dimensional fractal measures on $\mathbb R^{d+1}$:
\begin{equation}\label{avr}
\|   \mathcal A_\gamma f \|_{L^p(d\mu)} \le C \langle \mu \rangle_\alpha^{\frac1p} \| f\|_{L^p_\sigma (\mathbb R^d)} .
\end{equation}
Here, $\|f\|_{L_\sigma^p(\mathbb R^d)} : = \| (1-\Delta)^{\frac \sigma 2}f\|_{L^p(\mathbb R^d)}$.
Let us set $p(d) = 4$, if $d=2$, and $p(d) = 4d-2$, if $d\ge 3$. 

\begin{thm}\label{smoothing}
Let $d\ge 2$, $0<\alpha \le d+1$, and $\mu \in \mathfrak C^{d+1}(\alpha)$.
Suppose that $\gamma$ is a smooth curve satisfying \eqref{torsion}. If $p >p(d)$, then the estimate \eqref{avr}
holds for
\begin{align*} 
\sigma > (d-1-\alpha)/p .
\end{align*}
\end{thm}

Considering a specific function and $\alpha$-dimensional measure, one can show that $\sigma \ge (d-1-\alpha)/p$ is optimal for \eqref{avr} to hold.   
See Proposition \ref{opt} with $m=1$.

\subsection*{Multi-parameter dilation}  The approach for curves can be further generalized to other submanifolds via the corresponding local smoothing estimates for the associated averaging operators. 
In particular, we can extend Theorem \ref{thm-main} to families of  curves generated by multi-parameter dilations.

Let $1 \le m \le d$ and  denote $\mathbf t= (t_1,\dots,t_m)$, $t_i \in (1,2)$.  For an onto map   $\omega$ from $\{1 ,\dots,d \}$ to $\{1, \dots, m\}$, 
we define a family of nondegenerate curves $\gamma_{\mathbf t}^\omega
$ generated by $m$ parameters by
\begin{align}\label{m-curve}
\gamma_{\mathbf t}^\omega(s)
= (t_{\omega(1)}\gamma_1(s),\dots, t_{\omega(d)} \gamma_d(s)).
\end{align}
For $F' \subset \mathbb R^d\times (1,2)^m$, we obtain a lower bound of Hausdorff dimension of 
\[
\Gamma^\omega(F')  =  \bigcup_{ (x, \mathbf t)    \in F'} \big( x+ \gamma_{\mathbf t}^\omega(I) \big) .   
\]

\begin{thm}\label{m-para} Let $d\ge2$, $2\le m\le d$,  and  $  0  <\alpha \le  d+m-2 $. Suppose $F' \subset \mathbb R^d\times (1,2)^m$ is a Borel set with $\dim_H F' >\alpha$.
Then, 
\begin{equation}\label{low-G}
\dim_H E \ge \max \{\alpha +2 - m, 0\}
\end{equation}
holds whenever  a Borel set $E$ contains  $\Gamma^\omega(F')$. 
\end{thm}

When $d=3$ and $m=2$, a special case was considered by  
Pramanik and Seeger \cite{PS}. It was shown  (\cite[Proposition 6.1]{PS})  that the 
union of a two-parameter family of helices $\gamma_{t_1,t_2}(s) = (t_1 \cos(2\pi s), t_1 \sin(2\pi s), t_2 s ) $ has Hausdorff dimension at least $8/3$.
Theorem \ref{m-para}, in particular, proves the optimal result  that  the union of $\gamma_{t_1,t_2}$ has Hausdorff dimension $3$. 
More generally, note that $\dim_H E = d$ if $\dim_H F' > d + m - 2$.  Furthermore, from Theorem \ref{m-smoothing}, one can show $E$ is of positive Lebesgue measure  if $\dim_H F' > d + m - 2$ (see Remark \ref{rmk} in p. 8). 
This is sharp in that, for $1\le m \le d$, there exists a set $E$ of Lebesgue measure zero containing $\Gamma^\omega(F')$ generated by a nondegenerate  curve $\gamma$ while $\dim_H F' = d+m-2$ (see Proposition \ref{BRK}).

\subsection*{Relation to local smoothing estimates for averages}
For $2 \le k \le d-1$, one may consider a similar packing  problem with $k$-dimensional submanifolds  under a suitable curvature condition.  
The problem for the case $k=d-1$ which includes spheres and hyperplanes is relatively well known 
(see \cite{KW, Mitsis, O06, O06', O07}).

Let $\mathfrak S\subset \mathbb R^d$ be a smooth compact $k$-dimensional submanifold.  Let $d\mathbf m$ denote the surface measure on $\mathfrak S$ and consider 
\[
\mathbf A f(x,t)= \chi(t)\int_{\mathfrak S} f(x + ty) \,d\mathbf m(y). 
\]
It is known that 
the averaging operator has smoothing properties under a suitable curvature condition on $\mathfrak S$. There are various model 
cases for which     $f\to \mathbf A f(\cdot,t)$ with a fixed $t\neq 0$ is  bounded from $L^p(\mathbb R^d)$  to $L_\sigma^p(\mathbb R^d)$ for some $\sigma<0$  and the best possible  regularity exponent is $\sigma=-k/p$ (e.g., see \cite{M, SSS, BGHS2, KLO}).  Taking additional integration over $t$ allows an extra regularity gain on a certain range of $p$. This phenomenon is known as \emph{local smoothing}.  It is not difficult to see the smoothing order can not  exceed $ (k+1)/p$ (e.g., see Section \ref{sec4.2}).  We refer to as sharp local smoothing estimate the  following:
\begin{equation}\label{submanifold}
\| \mathbf A f \|_{L^p (\mathbb R^{d+1})} \lesssim \|f\|_{L^p_\sigma(\mathbb R^d)}, ~\quad \sigma >- (k+1)/p. 
\end{equation}
If $k=1$, \eqref{submanifold} has been established in \cite{PS, GWZ, KLO} under the assumption that the curve   satisfies \eqref{torsion}.
For $k=d-1\ge 2$, the estimate \eqref{submanifold} is equivalent to Sogge's local smoothing conjecture when $\mathfrak S$ is strictly convex (\cite{BD15}) and 
similar results are also known when $\mathfrak S$ has non-vanishing gaussian (\cite{BD17, BHS}) with principal curvatures of different signs.

We now intend to draw a connection between the sharp local estimate \eqref{submanifold} and  the problem determining dimension of  union of $(x+t\mathfrak S)$. 
The following gives optimal lower bounds on dimension of  union of $(x+t\mathfrak S)$. It can be shown by adapting the proof of Theorem \ref{thm-main}, so we state it without proof.

\begin{prop} Let $1\le k\le d-1$. Suppose that  \eqref{submanifold} holds for some $p\neq \infty$. 
Suppose that a Borel set  $F\subset \mathbb R^d\times (0,\infty)$ satisfies $\dim_H F > \alpha$ for some $0<\alpha \le d-k $. Then, 
if  a Borel set $E\subset \mathbb R^d$ contains 
\[ \bigcup_{(x,t) \in F} (x+t\mathfrak S),\]
then $\dim_H  E\ge  k+\alpha$.  
\end{prop}

\subsection*{Generalized averaging operators}
Theorem \ref{thm-main} can also be extended to an analogous result for unions of variable curves in $\mathbb R^2$,
and more generally variable hypersurfaces in $\mathbb R^d$, $d\ge2$.

Let $U_\circ, V_\circ\subset \mathbb R^d$ be small open neighborhoods of $x_\circ$ and $y_\circ$ in $\mathbb R^d$, respectively.
Also, let $\Phi_t \in C^\infty( U_\circ \times V_\circ)$ vary smoothly in the parameter $t \in J_\circ := [1,2]$.
We set $U\subset\subset U_\circ$, $V \subset \subset V_\circ$, and $ J \subset \subset  J_\circ$.
For  $(x,t)\in U\times J$, we consider a family of variable hypersurfaces given by
\[
\mathcal G_{x,t}=\{ y\in V: \Phi_t(x,y)= 0 \}.
\]
We assume that $\Phi_{t}$ satisfies the rotational curvature condition: 
\begin{align}\label{S1}
\det
\begin{pmatrix}
\Phi_t 
& \partial_y \Phi_t \\
\partial_x \Phi_t & \partial_{xy}^2 \Phi_t
\end{pmatrix}
\neq0
\qquad \text{where} \quad \Phi_t(x,y)=0 .
\end{align}
Then, $\partial_y\Phi_t\neq0$ for any $(x,y) \in U_\circ\times V_\circ$, and $\mathcal G_{x,t}$ is a smooth embedded hypersurface in $\mathbb R^d$.
Furthermore, there exists a homogeneous function $q$ of degree $1$ in $\xi$ such that
$q$ is smooth if $\xi\neq0$ and satisfies
$q(x,t, \partial_x \Phi_t(x,y)) = \partial_t \Phi_t(x,y)$ whenever $\Phi_t(x,y)=0$
 (see \cite{Sogge2} for detailed construction of $q$).  We assume that
\Be\label{S2}
\textrm{rank}\, \partial_{\xi \xi}^2 q(x,t,\xi) =d-1~\text{ for }~  (x,t,\xi)\in U_\circ\times J_\circ\times (\mathbb R^d\setminus \{0\}).
\Ee
We say that $\Phi_t$ satisfies cinematic curvature condition if \eqref{S1} and \eqref{S2} hold.

Proving a sharp estimate for a maximal average over a neighborhood of $\mathcal G_{x,t}$ in $\mathbb R^2$,
Zahl \cite{Za} showed that a Borel set containing curves $\mathcal G_{x,t}$ for every $0<t<1$ has Hausdorff dimension $2$ (see also \cite{KW}). 
By utilizing the local smoothing estimates in \cite{GLMX, BHS}, we obtain the following.

\begin{thm}\label{vari}
Let $d\ge2$ and $0<\alpha < 1$. Suppose $F\subset U\times J$ is a Borel set satisfying $\dim_H F>\alpha$ and  $\Phi_t$ satisfies the cinematic curvature condition. Then a Borel set containing $\mathcal G_{x,t}$
for $(x,t)\in F$ has Hausdorff dimension at least $d-1+\alpha$.
\end{thm}

\subsubsection*{Organization of the paper}
In Section \ref{thm-proof}, we prove Theorem \ref{thm-main} and \ref{m-para} by using local smoothing type estimates and Bessel capacity. 
In Section \ref{variable}, we provide a proof of Theorem \ref{vari}. 
Finally, we discuss sharpness of Theorem \ref{m-para} and \ref{m-smoothing} in Section \ref{sharpness}.

\section{Proof of Theorem \ref{thm-main} and \ref{m-para}}
\label{thm-proof}
In this section we prove Theorem \ref{thm-main} and \ref{m-para} by applying Theorem \ref{smoothing} combined with Proposition \ref{prop}  which shows  that Hausdorff dimension of unions of curves can be determined in terms of the regularity of the smoothing estimate  \eqref{avr}.

\subsection{Local smoothing estimates relative to fractal measure}
In this section, we prove Theorem \ref{smoothing}.
In fact, we prove a slightly general result considering multi-parameter dilation.
Let $1 \le m \le d$ and $\mathbf t =(t_1,\dots,t_m)$.
Recall that $\gamma_{\mathbf t}^\omega$ is given by \eqref{m-curve}.
We consider
\[\mathcal A_{\gamma}^\omega f (x, \mathbf t) = \prod_{j=1}^m \chi( t_j) \int f(x +\gamma_{\mathbf t}^\omega(s))\psi(s)\,ds \]
where $  \chi \in C_0^\infty((1/2,4))$ such that $ \chi=1$ on $[1,2] $ and $\psi$ is a real valued smooth function supported in $I_\circ$.  
Theorem \ref{smoothing} follows from the next theorem with $m=1$. 
\begin{thm}\label{m-smoothing}
Let $d\ge 2$ and $1\le m \le d$. Also let $0<\alpha \le n:= d+m$ and $\mu \in \mathfrak C^{n}(\alpha)$.
Suppose that $\gamma$ is a smooth curve satisfying \eqref{torsion}. If $p >p(d)$, then 
\begin{equation}\label{m-ls}
\big\|  \mathcal A_{\gamma}^\omega   f \big\|_{L^p(\mathbb R^{n},\, d\mu)} \le C\langle \mu \rangle_\alpha^{\frac 1p}   \|f\|_{L_\sigma^p} 
\end{equation}
for $\sigma > ( n -\alpha-2)/p$. 
\end{thm}
If $\mu$ is the Lebesgue measure, i.e., $\alpha=n$, the smoothing order $\sigma > -2/p$ does not depend on $m$. 
We deduce Theorem \ref{m-smoothing} from the following sharp local smoothing estimates for $\mathcal A_\gamma$. 
\begin{thm}[\cite{GWZ, KLO}] \label{thm-ls}
For $d\ge2$, let $\gamma$ be a smooth curve in $\mathbb R^d$ satisfying \eqref{torsion}. 
Then, for $p >p(d)$ and  any $ r > -2/p$, we have
\begin{equation}\label{ls}
\|  \mathcal A_\gamma f\|_{L^p(\mathbb R^{d+1})} \le C \|f \|_{L^p_{r}}.
\end{equation}
\end{thm}

We also use the next lemma which is commonly used to obtain estimates relative to fractal measures.

\begin{lem}[cf. Lemma 2.7 in \cite{HKL}]\label{fractal}
Let $\mu \in \mathfrak C^{n}(\alpha)$ for $0<\alpha \le n$. Suppose $\widehat F$ is supported on $\mathbb B^{n}(0,\lambda)$. Then 
$\| F\|_{L^p(d\mu)} \lesssim \langle \mu\rangle_\alpha^{\frac1p}\lambda^{\frac{n-\alpha}p} \|F\|_{L^p(\mathbb R^{n})}$ for $p\ge1$.
\end{lem}

\begin{proof}[Proof of Theorem \ref{m-smoothing}]
Let $\beta \in C_0^\infty((2^{-1},2))$ such that $\sum_j \beta(2^{-j}s)=1$ for $s \neq0$
and set $\beta_0=\sum_{j \le 0} \beta(2^{-j}\cdot)$.
Let $P_\lambda$ be a standard Littlewood-Paley projection operator given by $\widehat{P_\lambda f}(\xi) = \beta(\lambda^{-1}|\xi|) \widehat f(\xi)$.  
Then, we have 
\[
 \mathcal A_\gamma^\omega  f (x,\mathbf t) =  \mathcal A_\gamma^\omega P_0 f (x,\mathbf t)  + \sum_{j\ge 1}  \mathcal A_\gamma^\omega P_{2^j} f (x,\mathbf t) ,
\]
where $\widehat{P_0f}=\beta_0(|\xi|)\widehat f(\xi)$.
Since the kernel of $ \mathcal A_\gamma^\omega P_0 f (x,\mathbf t)$ decays rapidly outside of the unit ball, we have
$ \|    \mathcal A_\gamma^\omega P_0 f  \|_{L^p(d\mu)} \le C \|f\|_{L^p}$ by Schur's test.
Thus it suffices to show that, for $\lambda \ge1$,
\begin{equation}\label{j}
\|  \mathcal A_\gamma^\omega P_\lambda f \|_{L^p(d\mu)} \le C \langle \mu\rangle_\alpha^{\frac1p} \lambda^\sigma \|f\|_{L^p}  ,~ \quad \sigma > (n-\alpha-2)/p.
\end{equation}

Since $\chi$ is compactly supported, the support of the space-time Fourier transform of  $ \mathcal A_\gamma^\omega P_\lambda f$  is unbounded. 
To apply Lemma \ref{fractal}, we decompose the frequency support of $ \mathcal A_\gamma^\omega P_\lambda f$ into a bounded set and its complement. 

For $ \tau=(\tau_1,\dots,\tau_m)$,  let us set
\[
\mathfrak m_\lambda (\xi,\tau) = \beta(\lambda^{-1}|\xi|)   \iint e^{2\pi i ( \xi\cdot \gamma_{\mathbf t}^\omega(s) -\mathbf t \cdot  \tau )} \prod_{j=1}^m \chi(t_j)\psi(s) \, ds\, d\mathbf t ,  
\]
so we have $\mathcal F_{x,\mathbf t} (\mathcal A_\gamma^\omega P_\lambda f) (\xi, \tau) =\mathfrak m_\lambda (\xi,\tau)\widehat f(\xi)$.  
We define a frequency localized operator $ \mathcal A_\gamma'$ by 
\[
\mathcal F_{x,\mathbf t} (\mathcal A_\gamma' f) (\xi,\tau) = \prod_{j=1}^m \beta_0 \big(   ( C\lambda)^{-1}\tau_j \big) \mathfrak m_\lambda (\xi,\tau)\widehat f(\xi) 
\]
for a sufficiently large constant $C\ge1$ to be chosen later.
We set 
\[ 
\mathcal A_\gamma^{e}   = \mathcal A_\gamma^\omega   - \mathcal A_\gamma'.
\]

First, we show 
\begin{align}\label{decay}
\| \mathcal A_\gamma^e P_\lambda f  \|_{L^p(d\mu)} \lesssim \langle \mu \rangle_\alpha^{\frac1p} \lambda^{-N} \|f\|_p 
\end{align}
for any $N\ge 1$.
We may write $A_\gamma^e P_\lambda f(x,\mathbf t) = \mathcal K_\lambda (\cdot,\mathbf t)\ast f(x) $, where 
\[
\mathcal K_\lambda(x,\mathbf t)= \iint \Big( 1- \prod_{j=1}^m \beta_0((C\lambda)^{-1}\tau_j)  \Big) \mathfrak m_\lambda (\xi,\tau) e^{2\pi i (x\cdot \xi + \mathbf t \cdot \tau)} d\xi d\tau.
\]
By integration by parts in $\mathbf t$, we have
$|\partial_{\xi,\tau}^{\kappa}\mathfrak m_\lambda(\xi,\tau)| \lesssim \lambda^{-N}(1+|\tau|)^{-N}$ for any multi-indices $\kappa$ and $N\ge1$,  provided that  $|\tau| \ge C|\xi|$
for a sufficiently large $C$ satisfying $C\ge1+10\sup_{s \in I_\circ} |\gamma(s)|$. 
Then, integration by parts yields
$
| \mathcal K_\lambda(x,\mathbf t) | \lesssim \lambda^{- N} (1+ |x|)^{-N} (1+|\mathbf t|)^{-N}
$
for any $N\ge 1$.
By Schur's test, we get the estimate \eqref{decay}  for any $\mu \in \mathfrak C^{n}(\alpha)$.

Now, we show 
\begin{equation}\label{loc}
\|   \mathcal A_\gamma'   P_\lambda f\|_{L^p(d\mu)} \lesssim \langle\mu\rangle_\alpha^{\frac1p} \lambda^{\frac{n-\alpha}{p} + r } 
\|    f\|_{L^p(\mathbb R^{d})}  , ~ \quad r > -2/p.
\end{equation}
Since $\mathcal F_{x,\mathbf t} (\mathcal A_\gamma' P_\lambda f)$ is supported in $\mathbb B^{n} (0,C (d,m)\lambda)$, we can apply Lemma \ref{fractal} to get 
\[
\|\mathcal A_\gamma' P_\lambda f\|_{L^p(d\mu)} \lesssim \langle\mu\rangle_\alpha^{\frac1p}\lambda^{\frac{n-\alpha}p}  \| \mathcal A_\gamma' P_\lambda f\|_{L^p(\mathbb R^{n})} .
\]  
Using $| \mathcal A_\gamma' P_\lambda f| \le |\mathcal A_\gamma^\omega P_\lambda f| + |\mathcal A_\gamma^e P_\lambda f|$ and \eqref{decay},
we see 
\[ 
\|   \mathcal A_\gamma' P_\lambda f\|_{L^p(d\mu)} \lesssim \langle\mu\rangle_\alpha^{\frac1p} \lambda^{\frac{n-\alpha}p} 
\|   \mathcal A_\gamma^\omega P_\lambda f\|_{L^p(\mathbb R^{n})} +\langle\mu\rangle_\alpha^{\frac1p}\lambda^{-N} \|f\|_p
\]  
for any $N\ge1$.
In order to prove \eqref{loc}, it suffices to show 
\begin{align}\label{loc2}
\|\mathcal A_\gamma^\omega P_\lambda f\|_{L^p(\mathbb R^{n})}  \lesssim \lambda^r \| f\|_{L^p}, ~ \ r >-2/p.
\end{align}
Using Fubini's theorem, we make change of variables $(t_1,t_2, \dots,t_m) \mapsto (t_1,t_1 t_2,\dots,t_1 t_m)$ to get
\begin{align*}
&\big\| \mathcal A_\gamma^\omega P_\lambda f\big\|_{p}^p
 = \int\!\!\cdots\!\!\int  \Big( \iint \big| \mathcal A_{\gamma_{\mathbf t'}^\omega} P_\lambda f (x, t_1 )  \big|^p t_1^{m-1}  \prod_{j=2}^m \chi(t_1t_j)\, dx \,d t_1 \Big) \, dt_2 \cdots  d t_d  ,
\end{align*}
where $\mathbf t' = (1,t_2,\dots,t_m)$. 
For fixed $t_2,\dots, t_m$, it is obvious that $\gamma_{\,\mathbf t'}^\omega$ is nondegenerate. 
Applying Theorem \ref{thm-ls} to the inner integral, we obtain \eqref{loc2}, and thus \eqref{loc} follows.

Combining \eqref{decay} and \eqref{loc}, we get
the desired estimate \eqref{j} for $\sigma =(n -\alpha)/p + r  > (n-\alpha-2)/p$.
\end{proof}

\subsection{Bessel capacity}
To prove Theorem \ref{thm-main} and \ref{m-para},
we begin with  Bessel capacity which is closely related to Hausdorff dimension. 
For a Borel set $ E \subset \mathbb R^d$, let us define the Bessel capacity $B_\sigma^{p}(E)$ by
\[
B_\sigma^{p}( E )= \inf \big\{ \|f\|_{L_\sigma^{p}}^p: f \in C_0^\infty(\mathbb R^d),\, f \ge \chi_E \big\} 
\]
for $ \sigma >  0 $ and $ p>1$.
By $\mathcal H^\sigma$ we denote $\sigma$-dimensional Hausdorff measure. The following describes a relation between the Bessel capacity and the Hausdorff measure. 
\begin{thm}[Theorem  2.6.16 in \cite{Zi}]\label{capacity}
Let $p>1$ and $0<\sigma p<d$. 
Also let $E \subset \mathbb R^d$ be a Borel set.
If $\mathcal H^{d-\sigma p}(E)<\infty$, then $B_\sigma^{p}(E)=0$.
Conversely, if $B_\sigma^{p}(E) = 0$, then $\mathcal H^{d-\sigma p+\epsilon}(E)=0$ for every $\epsilon>0$.
\end{thm}

Using \eqref{m-ls} and Theorem \ref{capacity}, we can obtain lower bounds of Hausdorff dimension of unions of curves.
\begin{prop}\label{prop}
Let $0<\alpha \le n:=d+m$ and $\mu \in \mathfrak C^{n}(\alpha)$. 
Suppose that the estimate \eqref{m-ls} holds  with $  \sigma >  \sigma_\circ$ for some $\sigma_\circ=\sigma_\circ( \alpha, n, p) >0 $ and $p>1$.  
If $F' \subset \mathbb R^d \times (1,2)^m$ is a Borel set with $\dim_H F' >\alpha$, then a Borel set
$E \subset \mathbb R^{d}$ containing $\Gamma^\omega(F')$ has Hausdorff dimension at least $d-   \sigma_\circ p$.
\end{prop}

\begin{proof}
Suppose that there exists a Borel set $E$ containing $\Gamma^\omega(F')$ such that
\[ \dim_H E <d- \sigma_\circ p .\] 
Taking a number $\sigma > \sigma_\circ$ such that $\dim_H E <d- \sigma p <d- \sigma_\circ p$, we see $\mathcal H^{ d - \sigma   p }(E) <\infty$. By Theorem \ref{capacity}, it follows 
\begin{equation}\label{zero}
B_{\sigma }^p (E) = 0 .
\end{equation}

For $\alpha_1$ satisfying $\alpha <\alpha_1<\dim_H F'$, there exists a compact set $F_1\subset F'$ such that $\mathcal H^{\alpha_1} (F_1) >0$. 
By Frostman's lemma (see \cite{Mattila} for example), there exists  $\mu \in \mathfrak C^{n}(\alpha_1)$ with $\supp \mu\subset F_1$. 
Let us consider a function  $f\in C_0^\infty(\mathbb R^d)$ such that $f \ge \chi_E$. Since $\supp f \supset E \supset \cup_{(x,\mathbf t)\in F_1} (x + \gamma_{\mathbf t}^\omega(I))$, we have $\mathcal A_\gamma^\omega f(x,\mathbf t) \gtrsim 1$ for all $(x,\mathbf t)\in F_1$. 
Since we are assuming  \eqref{m-ls} for $\sigma >\sigma_\circ$, it follows that
\begin{equation}\label{ineq}
\mu (F_1)^{\frac1p} \lesssim \|  \mathcal A_\gamma^\omega f \|_{L^p(d\mu)} \le C \langle \mu \rangle_{\alpha_1}^{\frac1p}  \| f \|_{L^p_\sigma (\mathbb R^d)}   
\end{equation}
for $ \sigma, p$ satisfying \eqref{zero}.
Note that \eqref{ineq} holds for any function $f\in C_0^\infty(\mathbb R^d)$ such that $f \ge \chi_E$. So, $\mu(F_1)^{\frac1p}  = 0$ from \eqref{zero}, which contradicts to $\supp \mu \subset F_1$. 
Hence $\dim_H E  \ge d- \sigma_\circ p $ for $E$ containing $\Gamma^\omega(F')$.
\end{proof}

Now we prove Theorem \ref{thm-main} and Theorem \ref{m-para}.

\begin{proof}[Proof of Theorem \ref{thm-main}]
Since $\dim_H F >\alpha$, there exists $\mu \in \mathfrak C^{d+1}(\alpha)$ which is supported in $F$ by Frostman's lemma. 
For such $\mu$, 
there exists $\sigma_\circ >0$ such that the estimate \eqref{avr} holds with $\sigma > \sigma_\circ > (d-1-\alpha)/p$ by Theorem \ref{smoothing}. 
Since $\alpha \le d-1$ and $p>1$,
by Proposition \ref{prop} with $m=1$ we see that a set $E\subset \mathbb R^d$ containing $\Gamma(F)$ has Hausdorff dimension at least $ d- \sigma_\circ p$. Taking $\sigma_\circ$ arbitrarily close to $(d-1-\alpha)/p$, we get \eqref{dimE}.
\end{proof}

\begin{proof}[Proof of Theorem \ref{m-para}]
Theorem \ref{m-para} can be obtained in the same manner. 
The difference is that we have \eqref{m-ls} with $\sigma > \sigma_\circ> (d+m -\alpha-2)/p$. If $\alpha \le d+m-2$, we can apply Proposition \ref{prop} to obtain $\dim_H E \ge d- \sigma_\circ p$. We omit the details.
\end{proof}

\begin{rmk}\label{rmk}
By using Theorem \ref{m-smoothing}, one can show that $E$ is of positive Lebesgue measure if $\dim_H F' > d + m - 2$. 
For $\alpha_1$ satisfying $\dim_H F' >\alpha_1 > d+m-2$, let $F_1$ be a compact subset of $F$ with $\mathcal H^{\alpha_1}(F_1) >0$.
Then there exists $\mu \in \mathfrak C^{d+m}(\alpha_1)$ with $\supp \mu \subset F_1$. 
Applying Theorem \ref{m-smoothing} with $\mu \in \mathfrak C^{d+m}(\alpha_1)$, we have 
\begin{equation*}
\mu(F_1)^{\frac1p} \lesssim \big\|  \mathcal A_{\gamma}^\omega  \chi_E \big\|_{L^p(\mathbb R^{d+m},\, d\mu)} \le C\langle \mu \rangle_{\alpha_1}^{\frac 1p}   \|\chi_E \|_{L_\sigma^p}  
\end{equation*}
for $\sigma > (d+m - \alpha_1-2)/p$.
Hence, $E$ can not be of Lebesgue measure zero.
\end{rmk}

\section{Dimension of unions of variable hypersurfaces}\label{variable}

In this section, we prove Theorem \ref{vari}  by making use of the sharp  local smoothing estimates for variable coefficient averaging operators.
We consider the Fourier integral operator given by
\[
\mathfrak A f (x,t) = \int a(x,t,y) \delta(\Phi_t(x,y)) f(y)\,dy 
\]
where $a \in C_c^\infty(U_\circ\times J_\circ\times V_\circ)$ is a positive bump function such that $a=1$ on $U\times J \times V$.
Sharp  local smoothing estimates for $\mathfrak Af$ were shown by Gao--Liu--Miao--Xi \cite{GLMX} for $d=2$, and  Beltran--Hickman--Sogge \cite{BHS} for higher dimensions.

\begin{thm}[{\cite{BHS, GLMX}}]\label{LS variable}
Suppose $\Phi_t$ satisfies the conditions \eqref{S1} and \eqref{S2} on the support of $a$. Then
\begin{align}\label{LSA}
\| \mathfrak A f\|_{L^p(\mathbb R^{d+1})} \lesssim \|f\|_{L_{r}^p(\mathbb R^d)}, \quad r> -d/p
\end{align}
holds for $4 \le p <\infty$ if $d=2$, and for $2(d+1)/(d-1) \le p <\infty$ if $d\ge3$.
\end{thm}

The range of $p$ can not be generally extended when $d=2$ and odd $d\ge3$.
In the same manner as in the proof of Theorem \ref{m-smoothing}, we obtain the following from Theorem \ref{LS variable}.

\begin{cor}\label{vari frac}
Let $0<\alpha \le d+1$ and $\mu \in \mathfrak C^{d+1}(\alpha)$. Then,
we have
\begin{equation}\label{vari claim}
\| \mathfrak A  f \|_{L^p(\mathbb R^{d+1},\, d\mu )} \le C\langle \mu \rangle_\alpha^\frac1p  \|f\|_{L_\sigma^p(\mathbb R^d)},
\quad \sigma > (1-\alpha)/p
\end{equation}
 for $4 \le p <\infty$ when $d=2$, and for $2(d+1)/(d-1) \le p <\infty$ when $d\ge3$.
\end{cor}

\begin{proof}
In local coordinates, modulo a smoothing operator,  $\mathfrak A f$ can be written as a finite sum of the operators 
\Be\label{fio}
\widetilde{\mathfrak A} f(x,t) =  \int e^{2\pi i\phi(x,t,\xi)} \mathfrak a(x,t,\xi)\widehat f(\xi)\,d\xi ,
\Ee
where $\mathfrak a \in C^\infty(\mathbb R^{d}\times \mathbb R\times \mathbb R^d)$ is compactly supported in $(x,t)$  and satisfies
\[|\partial_{x,t}^{\kappa_1}\partial_{\xi}^{\kappa_2} \mathfrak a(x,t,\xi)| \lesssim (1+|\xi|)^{-\frac{d-1}2-|\kappa_2|}\]  for any $\kappa_1$, $\kappa_2$. 
Also, $\phi$ is a smooth function on $\mathbb R^{d}\times \mathbb R\times(\mathbb R^d\setminus \{0\})$ of homogeneous of degree $1$ in $\xi$ satisfying the cinematic curvature condition on $\supp \mathfrak a$ (see \cite{H, MSS}).
So,  it is enough to show that
\begin{equation}\label{claim-1}
\| \widetilde{\mathfrak A}  P_\lambda f \|_{L^p( d\mu )} \le C \langle \mu \rangle_\alpha^\frac1p \lambda^{\sigma}   \|f\|_{p}
\end{equation}
holds for $\lambda \ge1$ if $\sigma > (1-\alpha)/p$.

By \eqref{LSA}, we have 
\begin{align}\label{LS tilde}
\| \widetilde{\mathfrak A} P_\lambda f\|_{L^p(\mathbb R^{d+1})} \lesssim \lambda^r \|f\|_{L^p(\mathbb R^d)}, \quad r >-d/p.
\end{align}
As in the proof of Theorem \ref{m-smoothing}, 
we write
\[ 
\widetilde{\mathfrak A}P_\lambda f=\widetilde{\mathfrak A}^{'}P_\lambda f+\widetilde{\mathfrak A}^{e}P_\lambda f  ,
\]
where
\begin{align*}
\mathcal F_{x,t} (\widetilde{\mathfrak A}^{'} P_\lambda f)(\zeta,\tau ) & =\beta_0 \big( (C\lambda)^{-1}|(\zeta,\tau)| \big) \mathcal F_{x,t} (\widetilde{\mathfrak A} P_\lambda f)(\zeta,\tau)  , \\
\mathcal F_{x,t} (\widetilde{\mathfrak A}^{e} P_\lambda f)(\zeta,\tau ) & =\big(1 - \beta_0 \big( (C\lambda)^{-1}|(\zeta,\tau)| \big) \big) \mathcal F_{x,t} (\widetilde{\mathfrak A} P_\lambda f)(\zeta,\tau)  
\end{align*}
for a sufficiently large $C\ge1$.

By rapid decay of the kernel, one can show 
\begin{align}\label{frak m}
\|\widetilde{\mathfrak A}^{e} P_\lambda f\|_{L^p(d \mu)} \lesssim \langle\mu\rangle_\alpha^{\frac1p}\lambda^{-N} \|f\|_p
\end{align}
for any $N\ge1$. Indeed, note that 
$\mathcal F_{x,t} (\widetilde{\mathfrak A} P_\lambda f)(\zeta,\tau)=\int \mathfrak w_\lambda (\zeta,\tau,y)f(y)\,dy$
where
\begin{align*}
\mathfrak w_\lambda (\zeta,\tau,y) = \iiint e^{-2\pi i(x'\cdot \zeta+t'\tau)} e^{2\pi i \phi(x',t',\xi)-2\pi iy\cdot \xi} \mathfrak a(x', t',\xi) \beta(\lambda^{-1}|\xi|) \, dx'dt'd\xi.
\end{align*}
So, we have
$ \widetilde{\mathfrak A}^{e} P_\lambda f(x,t) = \int \mathfrak K_\lambda(x,t,y) f(y)\,dy,$
where
\[\mathfrak K_\lambda (x,t,y) = \iint e^{2\pi i(x\cdot \zeta+t \tau)} \big(1-\beta_0( (C\lambda)^{-1}|(\zeta,\tau)|) \big) \mathfrak w_\lambda (\zeta,\tau,y) \,d\zeta d \tau.
\]
Note that $|\partial_{x,t}^{\kappa_1} \partial_{\xi}^{\kappa_2} \phi| \le C\lambda^{1-|\kappa_2|}$ by homogeneity of $\phi$. Thus repeated integration by parts in $x', t',\xi$  gives $|\partial_{\zeta,\tau}^\kappa \mathfrak w_\lambda (\zeta,\tau,y)| \lesssim \lambda^{-N}(1+|(\zeta,\tau)|)^{-N}(1+\lambda|y|)^{-N}$ for any $N\ge1$ if  $|(\zeta,\tau)| \ge C\lambda$ for a sufficiently large $C\ge 1$.
So, integration by parts in $\zeta, \tau$ gives 
\[
|\mathfrak K_\lambda (x,t,y)| \lesssim \lambda^{-N}(1+ |(x,t)|)^{-N} (1+\lambda|y|)^{-N}
\]
for any $N\ge1$, from which we obtain \eqref{frak m}.

To deal with  $\widetilde{\mathfrak A}^{'}P_\lambda f$, we apply Lemma \ref{fractal}, \eqref{frak m}, and \eqref{LS tilde} successively. 
So,  if   $r >-d/p$, we have 
\begin{align*} 
\|  \widetilde{ \mathfrak A}^{'} P_\lambda f\|_{L^p(d\mu)}  \lesssim \langle\mu\rangle_\alpha^{\frac1p} \lambda^{\frac{d+1-\alpha}p + r }  \|  f\|_{p} +\langle\mu\rangle_\alpha^{\frac1p} \lambda^{-N} \|f\|_p 
\end{align*}  for any $N\ge1$.
By this and \eqref{frak m}, we get \eqref{claim-1} for $\sigma > (1-\alpha)/p$. 
\end{proof}

Combining Corollary \ref{vari frac} and Proposition \ref{prop}, we  prove Theorem \ref{vari}.

\begin{proof}[Proof of Theorem \ref{vari}]
The proof is similar to that of Theorem \ref{thm-main}. 
By following the proof of Proposition \ref{prop}, one can easily see Proposition \ref{prop} with $m=1$ holds  for $\cup_{(x,t) \in F}\, \mathcal G_{x,t}$ which replaces $\Gamma^\omega(F')$. Indeed, taking $f \ge \chi_E$ such that $E \supset \cup_{(x,t)\in F} \mathcal G_{x,t}$, since $a =1$ on $U\times J \times  V$,
we see $\mathfrak Af(x,t)\gtrsim 1$ holds whenever $(x,t)\in F$. So, by applying \eqref{vari claim} in place of \eqref{m-ls}, we obtain 
\[
\mu (F)^{\frac1p} \lesssim \|  \mathfrak A f \|_{L^p(d\mu)} \le C \langle \mu \rangle_\alpha^{\frac1p}  \| f \|_{L^p_\sigma (\mathbb R^d)}   
\]
which replaces \eqref{ineq}. In the same manner, it follows that Proposition \ref{prop} holds for  $\cup_{(x,t) \in F}\mathcal G_{x,t}$ replacing $\Gamma^\omega(F')$.

By Corollary \ref{vari frac}, there is a $\sigma_\circ$ such that \eqref{vari claim} holds for $\sigma>\sigma_\circ>(1-\alpha)/p$. Combining this with Proposition \ref{prop} for $\cup_{(x,t) \in F}\, \mathcal G_{x,t}$,
we see that the set $E$ containing $\cup_{(x,t) \in F}\,\mathcal G_{x,t}$ has Hausdorff dimension at least $d- (1-\alpha)$ if $\dim_HF>\alpha$.
\end{proof}

\section{Sharpness of Theorem \ref{m-para} and \ref{m-smoothing}}\label{sharpness}
In this section, we discuss  optimality of the exponent $\alpha=d+m-2$ in Theorem \ref{m-para} (and Theorem \ref{thm-main}), and sharpness of  the regularity exponent $s$ in Theorem \ref{m-smoothing} (and Theorem  \ref{smoothing}).

\subsection{Construction of a type of BRK set for  space curves}
We show that the condition $\dim_H F' > d+ m -2$  in Theorem \ref{m-para} ($\dim_H F' > d-1$ in Theorem \ref{thm-main})  is necessary, in general, for $\Gamma  (F')$ to have positive Lebesgue measure.
Here, we may assume that $F'$ is a subset of $\mathbb R^d \times (0, 1)^m$. 

\begin{prop}\label{BRK}
Let $1\le m \le d$, and let $\gamma$ be a curve defined  on $I_\circ$ by $\gamma( s) = (s^{a_1},\dots, s^{a_d})$ for integers $1\le a_1<\dots<a_d$. Then, there exists a Borel set $F' \subset \mathbb R^d \times (0, 1)^m$ with $\dim_H F' = d + m -2$ such that  $\Gamma^\omega(F') $ is of Lebesgue measure zero.
\end{prop}

To prove Proposition \ref{BRK}, we use Kinney's  construction of Lebesgue measure zero set which consists of similitudes of  a planar  convex curve (\cite{Ki}).  When $m\le d-1$, as is clear from  the proof of Proposition \ref{BRK} below,  Proposition \ref{BRK} remains valid as long as  a projection of  $\gamma$ to 2-plane is convex. Thus, by a suitable change of coordinates  and Taylor expansion, one can see that Proposition \ref{BRK} continues to hold if $\gamma$ is of finite type.  
However, when $m=d$,  we need an additional assumption that at least two components of $\gamma$ are monomial.

\begin{proof}[Proof of Proposition \ref{BRK}]
Let $\gamma^{a,b} (s) = (s^a,s^b)$ for $ s\in I_\circ$ and $1\le a<b$.
Following the argument in \cite{Ki}, we first construct a compact set $K^{a,b}\subset \mathbb R^2$ of Lebesgue measure zero such that $K$ contains $t \gamma^{a,b}(I)$ for every $0 < t \le 1$.
Let $\mathcal C$ be the standard $1/3$-Cantor set on $[0,1]$. 
Note that $[0,1]$ is the distance set of $\mathcal C$. So, for each $t\in (0,1]$, there exist the smallest numbers $\mathfrak c_1 (t),  \mathfrak c_2(t) \in \mathcal C$  such that $\mathfrak c_2(t) - \mathfrak c_1(t) = t$.  Then the curve $(\mathfrak c_1(t), \mathfrak c_1(t)) + t\gamma^{a,b}(I)$ connects two points $(\mathfrak c_1(t), \mathfrak c_1(t))$ and $(\mathfrak c_2(t), \mathfrak c_2(t))$. 
Since $\gamma^{a,b}$ is a convex curve, 
adapting the argument in \cite{Ki}, one can show that the union of curves 
\begin{equation*}
K^{a,b} := \bigcup_{0< t\le 1}\Big( (\mathfrak c_1(t), \mathfrak c_1(t)) + t\gamma^{a,b}(I) \Big)
\end{equation*}
is of Lebesgue measure zero in $\mathbb R^2$ (see \cite[p.1080]{Ki}).

We treat the cases $m \le d-1$  and $m=d$, separately.
We consider the case $m\le d-1$ first. Since $m\le d-1$, $\omega(i)=\omega(j)$ for some $i\neq j$. There exists a dilation parameter, which we denote by $t_1$,   appearing at least  twice in $\gamma_{\mathbf t}^\omega=(t_{\omega(1)}\gamma_1(s),\dots, t_{\omega(d)} \gamma_d(s))$.  We may assume $i=1$, $j=2$, that is to say,    $t_{\omega(1)}=t_{\omega(2)}=t_1$.
We set \[ F' =  \big\{ (x,\mathbf t) \in [0,1]^d \times (0,1)^m : (x_1,x_2)=(\mathfrak c_1(t_1),\mathfrak c_1(t_1)) \big \}.\]
Obviously, $\dim_H F' =d+ m -2$ and
\begin{align*}
\Gamma^\omega(F') =\bigcup_{(x,\mathbf t) \in F'} ( x + \gamma_\mathbf t ^\omega(I))  \subset K^{a,b} \times [-C,C]^{d-2}
\end{align*}
for some constant $C>0$ where $(a,b)=(a_i, a_j)$ for some $i<j$. Therefore, it is clear that  $\Gamma^\omega(F')$ is of Lebesgue measure zero in $\mathbb R^d$. 

Now we consider the case $m=d$, for which we make use of homogeneity of monomials. 
We set \[ g_{\mathbf t}(s)=\gamma_{\mathbf t}^\omega((t_1/t_2)^{1/(a_2-a_1)}s).\]
The first two components of $g_{\mathbf t}(s)$ are given by $b(\mathbf t)s^{a_1},$ $b(\mathbf t)s^{a_2}$,  where $b(\mathbf t)=(t_1^{a_2}/t_2^{a_1})^{1/(a_2-a_1)}$.
Set $J_*= \{ (t_1, t_2)\in (0,1]^2: t_1\ge t_2, \, t_1^{a_2}\le t_2^{a_1} \}$ and consider 
\[F'=\big\{ (x,\mathbf t) \in [0,1]^d \times J_*\times (0,1)^{m-2}: (x_1,x_2)=(\mathfrak c_1(b(\mathbf t)), \mathfrak c_1(b(\mathbf t)))\big\}.\] 
Then, we have
\[
\bigcup_{(x,\mathbf t) \in F'} \big(x+g_{\mathbf t}(I) \big)
\subset
\bigcup_{(t_1,t_2) \in J_*} \Big( (\mathfrak c_1(b(\mathbf t)),\mathfrak c_1(b(\mathbf t))) +b(\mathbf t) \gamma^{a_1,a_2}(I) \Big)\times[-C,C]^{d-2}
\]
for some $C>0$. Note that  
$0< b(\mathbf t)<1$ and $\gamma_{\mathbf t}^\omega(I) \subset g_{\mathbf t}(I)$ if $\mathbf t \in J_*\times(0,1)^{m-2}$. 
Thus, it follows that $ \Gamma^\omega(F') \subset K^{a_1, a_2} \times[-C,C]^{d-2}$. So,  $\Gamma^\omega(F')$ is of Lebesgue measure zero.
\end{proof}

\subsection{Optimality of the regularity exponent in Theorem \ref{m-smoothing}}
\label{sec4.2}
We prove that the regularity exponent in Theorem \ref{m-smoothing} is sharp for some range of $\alpha$.
Recall that  $\gamma_{\mathbf t}^\omega(s)=(t_{\omega(1)} \gamma_1(s),\dots, t_{\omega(d)} \gamma_d(s))$
for an onto map $\omega: \{1,\dots,d\} \rightarrow \{1, \dots, m\}$.
We set 
\[ 
 k(\omega)  = \max_{1 \le j \le m} \big|\{ i \in [1,d]: \omega(i) = j \}\big|
\] 
which is the maximum number of repetition  in $\{\omega(1),\dots, \omega(d)\}$.
For example, let $\gamma(s)=(\cos (2\pi s), \sin (2\pi s), s)$. If $\omega(1)=\omega(2)=1$ and $\omega(3)=2$,
then $\gamma_{t_1,t_2}(s) = (t_1 \cos(2\pi s), t_1 \sin(2\pi s) , t_2 s)$ is a two parameter family of helices with $k(\omega) = 2$.

\begin{prop}\label{opt}
Let $d\ge2$, $1 \le m \le d$, and let $\gamma$ be a finite type curve in $\mathbb R^d$.  
Set  $n=d+m$. For $0<\alpha \le n$, suppose \eqref{m-ls} holds for all $\mu \in \mathfrak C^{n}(\alpha)$. 
Then, if  $m\le d-1$,  we have
\begin{equation}\label{lowers}
\sigma \ge
\begin{cases}
\, (n-\alpha-2)/p, \qquad \ &\text{if~} \hfill  \quad    n- k (\omega)-1 < \alpha \le n,\\[0.8ex]
\, (k(\omega)-1)/p, \qquad \ &\text{if~} \hfill \quad   0<\alpha  \le n-k(\omega)-1. 
\end{cases}
\end{equation} 
If $m=d$ and $\gamma( s) = (s^{a_1},\dots, s^{a_d})$ for $1\le a_1<\dots<a_d$ and $s\in I_\circ$, then 
\begin{equation}\label{md}
\sigma \ge
\begin{cases}
\, (n-\alpha-2)/p, \qquad \ &\text{if~} \hfill  \quad    n- 3 < \alpha \le n, \phantom{~-k(\omega)}\\[0.8ex]
\, 1/p, \qquad \ &\text{if~} \hfill \quad    0 < \alpha  \le n-3.~\phantom{-k(\omega)} 
\end{cases}
\end{equation} 
\end{prop}

\begin{proof} 
Let $d$, $m$, and $k=k(\omega)$ be fixed, and denote $\widetilde x=(x_1,\dots,x_k)$  and $\overline x = (x_{k+1},\dots, x_d)$.

We consider the case $1 \le m \le d-1$ first. Note that $k\ge2$.
Without loss of generality, we may assume that $t_1$ is the parameter repeated $k$-times. 
Also, by a permutation, we may assume that $ ( t_{ \omega(1) },\dots,t_{ \omega(d)}  ) = (t_1 ,\dots,t_1, t_2,\dots,t_m)$.
So, let us write
\[  x +  {\gamma}_{\mathbf t}^\omega(s) =   (\widetilde x +t_1 \widetilde\gamma(s) ,~ \overline x + \overline {\gamma _{ \mathbf t  }^{ \omega  }} (s) )  \in \mathbb R^k\times \mathbb R^{d-k} .
\]

Since $\widetilde \gamma$ is of finite type in $\mathbb R^k$, by Taylor's expansion $\widetilde \gamma(s)=\sum_{j=1}^k \widetilde{\gamma}^{(\ell_j)}(0)s^{\ell_j}+O(s^{\ell_k+1})$ such that $\widetilde{\gamma}^{(\ell_j)}(0) \neq 0$ for some integers $1 \le \ell_1 <\dots<\ell_k$ if $|s|\le \delta_1$ for a sufficiently small $\delta_1$.
By a change of variables, we may assume that there is an interval $I_*=[\delta_0,2\delta_0] \subset [0,1]$ such that $|\widetilde \gamma'(s)| \ge c$ on $I_*$ for a constant $c>0$.

Let $\eta \in C_0^\infty([-2,2]^k)$ such that $\eta\ge0$ and $\eta=1$ on $[-1,1]^k$.
We also choose a positive function $h \in \mathcal S(\mathbb R^{d-k})$ such that $\supp \widehat h \subset [-1,1]^{d-k}$ and $ h \ge1$ on $[-C,C]^{d-k}$ for $C\ge1+10\sup_{s \in I_\circ} |\gamma(s)|$.
For a fixed $\lambda \gg \delta_0^{-1}$, we set
\[
g(\widetilde x)= \sum_{\nu  \in \lambda^{-1}\mathbb Z \cap I_*} \eta \big(\lambda( \widetilde x -  \widetilde {\gamma}(\nu)) \big)
\]
and $f(x)= g(\widetilde x) h( \overline x )$.
Let 
\[ Q = \big\{(x, \mathbf t) \in [-\varepsilon_0,\varepsilon_0]^d \times [2^{-1},2]^{m}: |\widetilde x| \le \varepsilon_0\lambda^{-1}, ~|t_1-1| \le \varepsilon_0\lambda^{-1} \big\}.
\] 
Then, we have
\begin{align}\label{A-low}
\mathcal A_{\gamma}^\omega f (x,\mathbf t) = \prod_{j=1}^m \chi( t_j) \int g(\widetilde x +t_1 \widetilde{\gamma}(s)) \, h (  \overline x + \overline{\gamma_{\mathbf t}^{\omega}}(s) ) \psi(s)\,ds \gtrsim 1, 
\end{align}
whenever $(x,\mathbf t) \in Q$ for a  sufficiently small $\varepsilon_0>0$. 
To see this, note that
\begin{align}\label{gg}  g (\widetilde x + t_1 \widetilde{\gamma}(s))=
\sum_{\nu  \in \lambda^{-1}\mathbb Z \cap I_*} \eta \big(\lambda ( \widetilde x + (t_1-1) \widetilde \gamma(s) + \widetilde \gamma(s) - \widetilde \gamma(\nu)) \big) . 
\end{align}
So, $g \big(\widetilde x + t_1 \widetilde \gamma(s)\big) \gtrsim1$ for $s \in I_*$ if $(x, \mathbf t) \in Q$. 
Also, note $g \big(\widetilde x + t_1 \widetilde \gamma(s)\big) \ge 0$, for $s \in [-1,1]\setminus I_*$, and $h \big(  \overline x +  \overline {\gamma_{\mathbf t}^\omega}(s)   \big)\gtrsim1$, for $s \in [-1,1]$ and $(x, \mathbf t) \in Q$. 
Hence, we get \eqref{A-low}.

Now, we show 
\[ \|f\|_{L_\sigma^p(\mathbb R^d)} \lesssim \lambda^{\sigma-(k-1)/p}.
\]
We observe that $\|f\|_{L_\sigma^p(\mathbb R^d)} \lesssim \|g\|_{L_\sigma^p(\mathbb R^k)} \|h\|_{L^p(\mathbb R^{d-k})} + \|g\|_{L^p(\mathbb R^k)}\|h\|_{L_\sigma^p(\mathbb R^{d-k})}$ for $\sigma \ge 0$, and similarly $\|f\|_{L_\sigma^p(\mathbb R^d)} \lesssim \|g\|_{L_\sigma^p(\mathbb R^k)} \|h\|_{L^p(\mathbb R^{d-k})}$ for $\sigma<0$.
Those can be shown by using the Mikhlin multiplier theorem.  
Since $\|h\|_{L_\sigma^p(\mathbb R^{d-k})}\lesssim 1 $ for any  $\sigma\in \mathbb R$,  it suffices to show that
\begin{align}\label{g-norm}
\|g\|_{L_\sigma^p(\mathbb R^k)} \lesssim \lambda^{\sigma-(k-1)/p}.
\end{align}

Recall $\beta$, $\beta_0$ which are defined in the proof of Theorem \ref{m-smoothing}.
We decompose $g=\sum_{j\ge0} g_j$ such that $\widehat {g_j}( \widetilde \xi )=\widehat g(\widetilde \xi) \beta_j(\lambda^{-1}|\widetilde \xi|)$ with $\beta_j =\beta(2^{-j}\cdot)$ for $j\ge 1$.
Note that 
\[
\big((1+|\cdot|^2)^{\sigma/2} \widehat {g_j}\big)^\vee(\widetilde x)= \sum_{\nu  \in \lambda^{-1}\mathbb Z \cap I_*} \mathfrak G_{j,\nu}(\widetilde x) ,
\]
where
\[
\mathfrak G_{j,\nu}(\widetilde x)=\int  e^{2\pi i (\widetilde x - \widetilde \gamma(\nu)) \cdot \widetilde \xi} \lambda^{-k}\widehat \eta (\lambda^{-1} \widetilde \xi) \beta_j(\lambda^{-1}|\widetilde \xi|)(1+|\widetilde \xi|^2)^{\sigma /2} \,d \widetilde \xi.
\]
By rescaling $\widetilde \xi \rightarrow \lambda \widetilde \xi$, it is easy to show that
$|\mathfrak G_{j,\nu}(\widetilde x)| \lesssim 2^{-jN}\lambda^\sigma (1+\lambda| \widetilde x -  \widetilde \gamma(\nu)|)^{-N}$ for any $N\ge1$.
Since $|\widetilde \gamma'(s)| \ge c$ on $I_*$ for a constant $c>0$, we have  $|\widetilde \gamma(\nu)-\widetilde \gamma(\nu')| \ge c|\nu-\nu'|$ for $ \nu, \nu'\in I_*$. 
Therefore, we see $\| \sum_\nu  \mathfrak G_{j,\nu} \|_p \lesssim 2^{-jN}\lambda^{\sigma-(k-1)/p}$ for  any $N\ge 1$.
This gives $\sum_{j\ge0} \|g_j\|_{L_\sigma^p}  \lesssim \lambda^{\sigma-(k-1)/p}$ and hence \eqref{g-norm}.

We now take
\begin{align*}
d\mu(x, \mathbf t)=
\begin{cases}
\lambda^{n-\alpha} \chi_{Q}(x,\mathbf t) \, dxd\mathbf t,
\quad \ &\text{if}\hfill \quad n-k-1< \alpha \le n, \\[1ex]
\lambda^{k+1} \chi_{Q}(x,\mathbf t) \, dx d\mathbf t,
\quad \ &\text{if}\hfill \quad  0<\alpha \le n-k-1.
\end{cases}
\end{align*}
It is easy to see that $\mu \in \mathfrak C^{n}(\alpha)$ and $\langle \mu \rangle_\alpha \lesssim 1$. Indeed, 
when $n-k-1< \alpha \le n$, we note that $\mu(\mathbb B^{n} (z,r)) \lesssim \lambda^{n-\alpha}r^{n}$ if $r \le \lambda^{-1}$, and $\mu(\mathbb B^{n}(z,r)) \lesssim \lambda^{n-\alpha-k-1}r^{n-k-1}$ if $r\ge \lambda^{-1}$. 
Similarly, when $0<\alpha \le n-k-1$, we have $\mu(\mathbb B^{n}(z,r)) \lesssim \lambda^{k+1}r^{n}$ if $r \le \lambda^{-1}$ and $\mu(\mathbb B^{n}(z,r)) \lesssim r^{n-k-1}$ if $r>\lambda^{-1}$.
Combining the two cases, we see $\mu(\mathbb B^{n}(z,r)) \lesssim r^\alpha$ as desired.

Note that $|Q| \sim \varepsilon_0^{d+1}\lambda^{-k-1}$ and  $\mu(Q)\gtrsim \varepsilon_0^{d+1}\min\{\lambda^{n-\alpha-k-1},1\}$. Therefore, 
\begin{align}\label{AA}
\min\{\lambda^{(n-\alpha-k-1)/p},1\} \lesssim \| \mathcal A_{\gamma}^\omega f\|_{L^p(d\mu)} \lesssim \|f\|_{L_\sigma^p(\mathbb R^d)} \lesssim \lambda^{\sigma-(k-1)/p}.
\end{align}
Letting $\lambda\to \infty$, we get \eqref{lowers}.

\smallskip

We now consider the case $m=d$ and $\gamma( s) = (s^{a_1},\dots, s^{a_d})$. As before, we exploit  homogeneity of monomials. 
We denote $\widetilde x$ and $\overline x$ same as above with $k=2$, i.e., $\widetilde x=(x_1,x_2)$  and $\overline x = (x_{3},\dots, x_d)$. 
Set
 $c(\mathbf t)=(t_1/t_2)^{1/(a_2-a_1)}$. By the change of variable $s \mapsto c(\mathbf t)s$,  we have
\[
\mathcal A_\gamma^\omega f(x,\mathbf t) =\prod_{j=1}^d \chi(t_j) \int
f\big(x- \gamma_{\mathbf t}^\omega (c(\mathbf t)s) \big) \psi (c(\mathbf t)s )c(\mathbf t)\,ds.
\]
%

Note that $ \widetilde{\gamma_{ \mathbf t}^\omega} (c(\mathbf t)s)=(t_1^{a_2}/t_2^{a_1})^{1/(a_2-a_1)}(s^{a_1},s^{a_2})$. Then, we can repeat the same argument as above. Let $f$ and $\mu$ be given in the same manner as before with $k=2$. 
We also set
\[ Q' = \{(x, \mathbf t) \in [-\varepsilon_0,\varepsilon_0]^d \times [2^{-1},2]^{d}: |\widetilde x| \le \varepsilon_0\lambda^{-1}, ~|(t_1^{a_2}/t_2^{a_1})^{1/(a_2-a_1)}-1| \le \varepsilon_0\lambda^{-1} \}.
\] 
Then $\mathcal A_\gamma^\omega f (x,\mathbf t) \gtrsim 1$ whenever $(x,\mathbf t) \in Q'$ for a sufficiently small $\varepsilon_0>0$.
Note that for each fixed $t_2$, we have $|t_1-t_2^{a_1/a_2}| \lesssim \varepsilon_0 \lambda^{-1}$ whenever  $(x,\mathbf t) \in Q'$.
Thus $|Q'| \sim \varepsilon_0^{d+1}\lambda^{-3}$ and hence $\mu(Q') \gtrsim \varepsilon_0^{d+1} \min\{\lambda^{n-\alpha-3},1\}$. This gives \eqref{md}.
\end{proof}

\subsection*{Acknowledgement} 
This work was supported by the NRF (Republic of Korea) grants  No. 2017R1C1B2002959 (Seheon Ham), No. 2022R1I1A1A01055527 (Hyerim Ko), and No.  2022R1A4A1018904 (Sanghyuk Lee and Sewook Oh).

\bibliographystyle{plain}

\begin{thebibliography}{} 

\bibitem{BGHS} D. Beltran, S. Guo, J. Hickman, A. Seeger, \textit{Sharp $L^p$ bounds for the helical maximal function}, arXiv:2102.08272.

\bibitem{BGHS2} \bysame, \textit{Sobolev improving for averages over curves in $\mathbb R^4$},
Adv. Math., {\bf 393}  (2021)  108089.


\bibitem{BHS} D. Beltran, J. Hickman, C. D. Sogge, \textit{Variable coefficient Wolff-type inequalities and sharp local smoothing estimates for wave equations on manifolds}, Anal. PDE, \textbf{13} (2020), 403--433.




\bibitem{BR} A. Besicovitch, R. Rado, \textit{A plane set of measure zero containing circumferences of every radius}, J. London Math. Soc., {\bf 43} (1968), 717--719.

\bibitem{Bourgain} J. Bourgain, \textit{Averages in the plane over convex curves and maximal operators}, J. Analyse. Math. {\bf 47} (1986), 69--85.



\bibitem{BD15} J. Bourgain, C. Demeter, \textit{The proof of the $l^2$ decoupling conjecture}, Ann. of Math. {\bf 182} (2015), 351--389.

\bibitem{BD17} \bysame, \textit{Decouplings for curves and hypersurfaces with nonzero Gaussian curvature}, J. Anal. Math. {\bf 133} (2017), 279--311.

\bibitem{Da} R. O. Davies, \textit{Another thin set of circles}, J. London Math. Soc. {\bf 5} (1972) 191--192.


\bibitem{Fa} K. J. Falconer, \textit{The geometry of fractal sets}, Cambridge University Press  1985.

\bibitem{GLMX} C. Gao, B. Liu, C. Miao, Y. Xi, \textit{Square function estimates and Local smoothing for Fourier integral operators}, arXiv:2010.14390.

\bibitem{GWZ} L. Guth, H. Wang, R. Zhang, \textit{A sharp square function estimate for the cone in $\mathbb R^{3}$}, Ann. of Math.  \textbf{192} (2020), 551--581.


\bibitem{HKL} S. Ham, H. Ko, S. Lee, \textit{Circular average relative to fractal measures}, To appear in Commun. Pure Appl. Anal., (2022), doi: 10.3934/cpaa.2022100.

\bibitem{H} L. H\"ormander, \textit{Fourier integral operators. $\mathrm I$}, Acta Math. \textbf{127} (1971), 79--183.

\bibitem{KOV} A. K\"{a}enm\"{a}ki, T. Orponen, L. Venieri, \textit{A Marstrand-type restricted projection theorem in $\mathbb R^3$}, 
arXiv:1708.04859v2.


\bibitem{Ki} J. Kinney, \textit{A thin set of circles}, Amer. Math. Monthly, {\bf 75} (1968), 1077--1081.

\bibitem{KLO1} H. Ko, S. Lee, S. Oh, \textit{Maximal estimate for average over space curve,}  Invent. Math., {\bf 228} (2022),  991--1035.  

\bibitem{KLO} \bysame, \textit{Sharp smoothing properties of averages over curves}, arXiv:2105.01628v4.

\bibitem{KW}  L. Kolasa, T. Wolff, \textit{On some variants of the Kakeya problem}, Pacific J. Math. \textbf{190} (1999),  111--154.  


\bibitem{Le} B. Lepson, \textit{On a problem of Peter Fenton and the distance set of the Cantor set}, Notices Amer. Math. Soc. {\bf 23} (1976) A-507.

\bibitem{Marstrand} J. Marstrand, \textit{Packing circles in the plane}, Proc. London Math. Soc. \textbf{55} (1987), 37--58.

\bibitem{Mattila} P. Mattila, \textit{Fourier analysis and Hausdorff dimension}, Cambridge University Press, Cambridge, United Kingdom 2015. 

\bibitem{Mitsis} T. Mitsis, \textit{On a problem related to sphere and circle packing}, J. London Math. Soc. {\bf 60} (1999), 501--516. 

\bibitem{M} A. Miyachi, \emph{On some estimates for the wave equation in $L^p$ and $H^p$},  J. Fac. Sci. Univ. Tokyo Sect. IA Math. {\bf 27} (1980),  331--354. 

\bibitem{MSS} G. Mockenhaupt, A. Seeger, C. D. Sogge, \textit{Local smoothing of Fourier integral operators and Carleson-Sj\"olin estimates}, J. Amer. Math. Soc. \textbf{6} (1993),  65--130. 

\bibitem{O06} D. Oberlin, \textit{Packing spheres and fractal Strichartz estimates in $\mathbb R^d$ for $d\ge3$}, Proc. Amer. Math. Soc.,  \textbf{134} (2006), 3201--3209. 

\bibitem{O06'} \bysame, \textit{Restricted Radon transforms and unions of hyperplanes}, Rev. Mat. Iberoamericana, \textbf{22} (2006), 977--992.

\bibitem{O07} \bysame, \textit{Unions of hyperplanes, unions of spheres, and some related estimates}, Illinois J. of Math. \textbf{51} (2007), 1265--1274.




\bibitem{PS} M. Pramanik, A. Seeger, \textit{$L^p$ regularity of averages over curves and bounds for associated maximal operators},  Amer. J. Math. \textbf{129} (2007),  61--103. 



\bibitem{SSS}  A. Seeger, C.D. Sogge, E.M. Stein, \emph{Regularity properties of Fourier integral operators}.
Ann. of Math.  {\bf 134} (1991), 231--251.


\bibitem{Sogge2}  \bysame, \textit{Fourier integrals in classical analysis}, Cambridge Tracts in Mathematics, Cambridge University Press, 2017.



\bibitem{Talagrand} M. Talagrand, \textit{Sur la mesure de la projection d'un compact et certaines familles de cercles},  Bull. Sci. Math. \textbf{104} (1980),  225--231. 


\bibitem{Wo97}  T. Wolff, \emph{A Kakeya type problem for circles}, Amer. J. Math. \textbf{119} (1997), 985--1026.

\bibitem{Wo00} \bysame, \textit{Local smoothing estimates on $L^p$ for large $p$}, Geom. Funct. Anal. \textbf{10} (2000), 1237--1288.

\bibitem{Za} J. Zahl, \textit{On the Wolff circular maximal function}, Illinois J. Math. \textbf{56} (2012),  1281--1295.

\bibitem{Zi} W. P. Ziemer, \textit{Weakly Differentiable Functions}, Graduate Texts in Mathematics, Springer-Verlag, New York, 1989.


 
\end{thebibliography}

\end{document}